\documentclass[12pt]{amsart}
\usepackage{amssymb,latexsym}

\newtheorem{theorem}{Theorem}

\newtheorem{proposition}[theorem]{Proposition}

\newcommand{\Cbb}{\mathbb{C}}
\newcommand{\Rbb}{\mathbb{R}}

\begin{document}

\title{A Simple Proof of Unique Continuation for $J$-holomorphic Curves}
\author{Michael VanValkenburgh}
\date{May 1, 2009}
\address{UCLA Department of Mathematics, Los Angeles, CA 90095-1555, USA}
\email{mvanvalk@ucla.edu}

\begin{abstract}
    In this expository paper, we prove strong unique continuation for $J$-holomorphic curves by first giving a simple proof of Aronszajn's theorem in the special case of the two-dimensional flat Laplacian.
\end{abstract}

\maketitle

\section{Introduction}

In the study of $J$-holomorphic curves and symplectic topology as presented by McDuff and Salamon \cite{R:McDSalJ}, a basic fact is the strong unique continuation property for $J$-holomorphic curves. In their book, the strong unique continuation property is a first step in a chain of events leading to the proof that, for a generic almost complex structure $J$, the moduli space $\mathcal{M}^{*}(A,\Sigma;J)$ of simple $J$-holomorphic $A$-curves is a smooth finite dimensional manifold, and from there to the construction of the Gromov-Witten invariants for a suitable class of symplectic manifolds (see pages 4 and 38 of \cite{R:McDSalJ} for the outline of this approach).

McDuff and Salamon give three proofs of the unique continuation
property. The first proof is a few lines long but cites
Aronszajn's theorem as proven in \cite{R:Aronszajn}. The second
and third proofs are given self-contained treatments, and,
moreover, the methods find further application in their book. The
second proof uses the Hartman--Wintner theorem \cite{R:HartWint}
(proven in McDuff and Salamon's Appendix~E.4), which in fact
implies the needed special case of Aronszajn's theorem, and the
third proof uses the Carleman similarity principle and the
Riemann-Roch theorem (proven in their Appendix~C).

Here we return to the first method of proof, but give a simplified argument. The method is well known in certain branches of partial differential equations; it is the method of weighted integral estimates depending on a parameter. This is also the approach of Aronszajn \cite{R:Aronszajn}, but it goes back even further, to Carleman \cite{R:Carleman}. For a general treatment with some historical comments, one may consult Sections~17.1 and 17.2 of H\"{o}rmander's book \cite{R:HoIII} or his corresponding paper \cite{R:HoUniq}. However, all these references give much more than is needed for our application. Here we present only what is needed for $J$-holomorphic curves.

We give the full details for the case of $C^{\infty}$ $J$-holomorphic curves; for $J$-holomorphic curves in Sobolev spaces with minimal assumptions, discussed in McDuff and Salamon's book, one may find the appropriate modifications in Sections 17.1 and 17.2 of H\"{o}rmander's book \cite{R:HoIII}. Here we focus on the $C^{\infty}$ case, for ease of exposition and since the $C^{\infty}$ case is sufficient for many purposes; after all, in Gromov's original definition all $J$-holomorphic curves are $C^{\infty}$ \cite{R:GromovJ}.

The weighted integral estimates will depend on a parameter $0<h\ll 1$ which may be interpreted as ``Planck's constant'', as appearing in the correspondence principle of the old quantum theory, or, more generally, as appearing in semiclassical analysis \cite{R:ZwSC}. The general idea is that as $h$ tends to zero, asymptotic analysis reveals the classical mechanics of the operator's symbol, interpreted as a Hamiltonian function. Hence symplectic geometry plays a role beneath the surface.

We begin by recalling the basic definitions, so that our presentation is self-contained. Let $(\Sigma,j)$ be a Riemann surface and $(M,J)$ an almost complex manifold. A smooth function $u:\Sigma \rightarrow M$ is called a \textbf{$J$-holomorphic curve} if its differential $du$ is a complex linear map with respect to $j$ and $J$; that is, if $$J\circ du =du\circ j,$$ or, equivalently,
$$\bar\partial_{J}(u):=\frac{1}{2}\left(du+J\circ du\circ j\right)=0.$$

Unique continuation is a local problem, so for our purposes we may take the domain of $u$ to be a connected neighborhood $X\subset\Cbb$ of the origin, writing the elements of $X$ as $x=x_{1}+ix_{2}$, and we may take $M$ to be $\Cbb^{n}$. Hence we are interested in those $u\in C^{\infty}(X,\Cbb^{n})$ satisfying
\begin{equation}\label{E:locdefJhol}
    \partial_{x_{1}}u+J(u)\partial_{x_{2}}u=0,
\end{equation}
where $J:\Cbb^{n}\rightarrow \text{GL}(2n,\Rbb)$ is, say, a $C^{1}$ function such that $J^{2}=-I$.

The main point of this paper is to give a simple, elementary proof of the following strong unique continuation result:
\begin{theorem}\label{T:uniqJhol}
    Let $X\subset\Cbb$ be a connected neighborhood of $0$, and suppose $u,v\in C^{\infty}(X,\Cbb^{n})$ satisfy (\ref{E:locdefJhol}) for some $C^{1}$ almost complex structure $J: \Cbb^{n}\rightarrow \text{GL}(2n,\Rbb)$. If $u-v$ vanishes to infinite order at $0$, then $u=v$ in $X$.
\end{theorem}

\vspace{12pt}

\noindent\textbf{Acknowledgements.} I thank Dusa McDuff and Dietmar Salamon for pointing out that the Hartman--Wintner theorem implies the special case of Aronszajn's theorem. Moreover, it is a pleasure to thank them for their excellent books.

\vspace{12pt}

\section{Proof of Unique Continuation}\label{S:proof}

Let $\Delta=\partial_{x_{1}}^{2}+\partial_{x_{2}}^{2}$ be the standard Laplacian. Since $(\partial_{x_{1}}J)J+J\partial_{x_{1}}J=0$, we have that any solution $u$ of (\ref{E:locdefJhol}) is also a solution of
\begin{equation*}
    \Delta u=(\partial_{x_{2}}J(u))\partial_{x_{1}}u-(\partial_{x_{1}}J(u))\partial_{x_{2}}u.
\end{equation*}
If $v$ is another such function, then
\begin{align*}
    \Delta(u-v)&=(\partial_{x_{2}}J(u))\partial_{x_{1}}(u-v)+[\partial_{x_{2}}(J(u)-J(v))]\partial_{x_{1}}v\\
        &\qquad\qquad -(\partial_{x_{1}}J(u))\partial_{x_{2}}(u-v)-[\partial_{x_{1}}(J(u)-J(v))]\partial_{x_{2}}v.
\end{align*}
Also, of course,
\begin{equation*}
    J(u)-J(v)=\int_{0}^{1}dJ(v+\tau(u-v))d\tau\cdot(u-v).
\end{equation*}
So, if $w:=u-v$, then for some constant $C>0$ we have
\begin{equation*}
    |\Delta w|\leq C(|w|+|\partial_{x_{1}}w|+|\partial_{x_{2}}w|).
\end{equation*}
Since we are considering fixed functions $u$ and $v$, the constant is allowed to depend on $u$, $v$, and their derivatives.

Thus Theorem~\ref{T:uniqJhol} is a consequence of the following unique continuation result, a special case of Aronszajn's theorem \cite{R:Aronszajn}. (We follow the presentation of Theorem~17.2.6 in H\"{o}rmander's book \cite{R:HoIII}.)

\begin{theorem}\label{T:uniqDiffIneq}
    Let $X\subset\Rbb^{2}$ be a connected neighborhood of $0$, and let $u\in C^{\infty}(X,\Cbb^{n})$ be such that
    \begin{equation}\label{E:DiffIneq}
        |\Delta u|\leq C(|u|+|\partial_{x_{1}}u|+|\partial_{x_{2}}u|).
    \end{equation}
    If $u$ vanishes to infinite order at $0$, then $u=0$ in $X$.
\end{theorem}

\begin{proof}
    For notational purposes, we assume $n=1$. The proof works line-by-line for the general case.

    We first introduce conformal polar coordinates in $\Rbb^{2}\backslash\{0\}$,
    $$(x_{1},x_{2})=(e^{t}\cos\theta,e^{t}\sin\theta)$$
    with $t\in\Rbb$ and $\theta\in S^{1}$. Then, in these coordinates,
    \begin{equation*}
        \partial_{x_{1}}=e^{-t}\cos\theta\,\partial_{t}-e^{-t}\sin\theta\,\partial_{\theta},
    \end{equation*}
    \begin{equation*}
        \partial_{x_{2}}=e^{-t}\sin\theta\,\partial_{t}+e^{-t}\cos\theta\,\partial_{\theta},
    \end{equation*}
    and
    \begin{equation*}
        \Delta=e^{-2t}(\partial_{t}^{2}+\partial_{\theta}^{2}).
    \end{equation*}

    Next, we ``convexify'' the coordinates. Let $0<\epsilon<1$, and let $T$ be such that $$t=T+e^{\epsilon T}.$$ As noted by H\"{o}rmander \cite{R:HoUniq}, this change of coordinates comes from the work of Alinhac and Baouendi \cite{R:AlinBao}. Then
    $$\frac{\partial t}{\partial T}=1+\epsilon e^{\epsilon T}>0,$$ and $T<t<T+1<T/2$ when $T<-2$. In these coordinates,
    \begin{equation*}
        \partial_{t}^{2}+\partial_{\theta}^{2}=(1+\epsilon e^{\epsilon T})^{-2}\partial_{T}^{2}-\epsilon^{2}(1+\epsilon e^{\epsilon T})^{-3}e^{\epsilon T}\partial_{T}+\partial_{\theta}^{2}.
    \end{equation*}
    Multiplying by $(1+\epsilon e^{\epsilon T})^{2}$, we get the operator
    \begin{equation*}
        Q:=\partial_{T}^{2}+c(T)\partial_{T}+(1+\epsilon e^{\epsilon T})^{2}\partial_{\theta}^{2},
    \end{equation*}
    with $$c(T):=-\epsilon^{2}(1+\epsilon e^{\epsilon T})^{-1}e^{\epsilon T}.$$

    \vspace{12pt}

    Our main tool is the following estimate:
    \begin{proposition}\label{P:Carlem}
        For some $T_{0}<0$ and some $h_{0}>0$ we have
        \begin{equation}\label{E:CarlemanEst}
        \begin{aligned}
            &h\iint\left(|U|^{2}+|h\partial_{T}U|^{2}+|h\partial_{\theta}U|^{2}+|h^{2}\partial_{T}^{2}U|^{2}
            +|h^{2}\partial_{T,\theta}^{2}U|^{2}+|h^{2}\partial_{\theta}^{2}U|^{2}\right)e^{-2T/h+\epsilon T}d\theta\,dT\\
            &\qquad\qquad\qquad\qquad\qquad\qquad\qquad\qquad\qquad\leq C\iint |h^{2}QU|^{2}e^{-2T/h}d\theta\,dT
        \end{aligned}
        \end{equation}
        for all $U\in C_{0}^{\infty}((-\infty,T_{0})\times S^{1})$, and for all $h\in (0,h_{0})$. (The constant $C>0$ is independent of $h$.)
    \end{proposition}
    \begin{proof} (of the Proposition.)
        We set $U:=e^{T/h}V$ and let $$\tilde{Q}:=h^{2}e^{-T/h}\circ Q\circ e^{T/h}.$$ That is,
        \begin{equation*}
            \tilde{Q}=(h\partial_{T}+1)^{2}+hc(T)(h\partial_{T}+1)+(1+\epsilon e^{\epsilon T})^{2}h^{2}\partial_{\theta}^{2}.
        \end{equation*}
        Then the estimate (\ref{E:CarlemanEst}) is equivalent to
        \begin{equation}\label{E:NewerCarlemanEst}
        \begin{aligned}
            &h\iint\left(|V|^{2}+|h\partial_{T}V|^{2}+|h\partial_{\theta}V|^{2}+|h^{2}\partial_{T}^{2}V|^{2}
            +|h^{2}\partial_{T,\theta}^{2}V|^{2}+|h^{2}\partial_{\theta}^{2}V|^{2}\right)e^{\epsilon T}d\theta\,dT\\
            &\qquad\qquad\qquad\qquad\qquad\qquad\qquad\qquad\qquad\leq C\iint |\tilde{Q}V|^{2}d\theta\,dT
        \end{aligned}
        \end{equation}
        for all $V\in C_{0}^{\infty}((-\infty,T_{0})\times S^{1})$.

        For bookkeeping purposes, we write $\tilde{Q}$ as the sum of its symmetric and antisymmetric parts, $$\tilde{Q}=A+B,$$ where
        \begin{equation*}
            A=h^{2}\partial_{T}^{2}+(1+hc-\frac{1}{2}h^{2}c^{\prime})+(1+\epsilon e^{\epsilon T})^{2}h^{2}\partial_{\theta}^{2},
        \end{equation*}
        and
        \begin{equation*}
            B=(2+hc)h\partial_{T}+\frac{1}{2}h^{2}c^{\prime}.
        \end{equation*}
        Hence, using the usual inner product notation on $L^{2}$, and with $[A,B]=AB-BA$ denoting the commutator,
        \begin{equation*}
            \iint |\tilde{Q}V|^{2}d\theta\,dT= ||AV||^{2}+||BV||^{2}+\langle[A,B]V,V\rangle.
        \end{equation*}
        Repeated integration by parts gives
        \begin{align}
            ||AV||^{2}&=||h^{2}\partial_{T}^{2}V||^{2}\notag\\
            &\qquad +||(1+hc-\frac{1}{2}h^{2}c^{\prime})V||^{2}\label{E:firstfriend}\\
            &\qquad +||(1+\epsilon e^{\epsilon T})^{2}h^{2}\partial_{\theta}^{2}V||^{2}\label{E:secondfriend}\\
            &\qquad +h^{3}\langle V,(c^{\prime\prime}-\frac{1}{2}hc^{\prime\prime\prime})V\rangle\notag\\
            &\qquad -2\langle h\partial_{T}V,(1+hc-\frac{1}{2}h^{2}c^{\prime})h\partial_{T}V\rangle\notag\\
            &\qquad -2\epsilon^{3}h^{2}\langle h\partial_{\theta}V,(1+2\epsilon e^{\epsilon T})e^{\epsilon T}h\partial_{\theta}V\rangle\notag\\
            &\qquad +2\langle h^{2}\partial_{T,\theta}^{2}V,(1+\epsilon e^{\epsilon T})^{2}h^{2}\partial_{T,\theta}^{2}V\rangle\notag\\
            &\qquad -2\langle h\partial_{\theta}V,(1+hc-\frac{1}{2}h^{2}c^{\prime})(1+\epsilon e^{\epsilon T})^{2}h\partial_{\theta}V\rangle,\label{E:Trouble}
        \end{align}

        \begin{align}
            ||BV||^{2}&=||(2+hc)h\partial_{T}V||^{2}-\frac{1}{4}h^{4}||c^{\prime}V||^{2}-\frac{1}{2}h^{4}\langle V,c^{\prime\prime}cV\rangle
            -h^{3}\langle V,c^{\prime\prime}V\rangle,\notag
        \end{align}

        and
        \begin{align}
            \langle [A,B]V,V\rangle
            &=-2h^{2}\langle c^{\prime}h\partial_{T}V,h\partial_{T}V\rangle\notag\\
            &\qquad -2h^{2}\langle c^{\prime}V,V\rangle+h^{3}\langle(c^{\prime\prime}-cc^{\prime})V,V\rangle +\frac{1}{2}h^{4}\langle(cc^{\prime\prime}+c^{\prime\prime\prime})V,V\rangle\notag\\
            &\qquad + 2h\epsilon^{2}\langle(2+hc)(1+\epsilon e^{\epsilon T})e^{\epsilon T}h\partial_{\theta}V,h\partial_{\theta}V\rangle.\label{E:commutatorfriend}
        \end{align}

        \vspace{12pt}

        Also, we recall that $$c(T)=-\epsilon^{2}e^{\epsilon T}(1+\epsilon e^{\epsilon T})^{-1}$$ so that
        $$c^{\prime}(T)=-\epsilon^{3}e^{\epsilon T}(1+\epsilon e^{\epsilon T})^{-2}$$ is also a negative quantity.

        Most of the terms in the above expansions may be absorbed into other terms when we take $0<h$ to be sufficiently small. It is only the term (\ref{E:Trouble}) that gives some difficulty. We write (\ref{E:Trouble}) as
        \begin{equation}\tag{\ref{E:Trouble}$'$}
            -2\langle h\partial_{\theta}V,(1+\lambda hc)(1+\epsilon e^{\epsilon T})^{2}h\partial_{\theta}V\rangle
            -2\langle h\partial_{\theta}V,((1-\lambda)hc-\frac{1}{2}h^{2}c^{\prime})(1+\epsilon e^{\epsilon T})^{2}h\partial_{\theta}V\rangle.
        \end{equation}
        Here $\lambda\in\Rbb$ is to be determined; as we will see, any $2<\lambda<3$ will suffice.

        For the first term of (\ref{E:Trouble}$^{\prime}$), we use the elementary inequality
        \begin{align}
            2\langle(1+\lambda hc)^{1/2}&V,(1+\lambda hc)^{1/2}(1+\epsilon e^{\epsilon T})^{2}h^{2}\partial_{\theta}^{2}V\rangle\notag\\
            &\geq -\langle(1+\lambda hc)V,V\rangle\label{E:firstguy}\\
            &\qquad\qquad -\langle (1+\lambda hc)(1+\epsilon e^{\epsilon T})^{2}h^{2}\partial_{\theta}^{2}V,(1+\epsilon e^{\epsilon T})^{2}h^{2}\partial_{\theta}^{2}V\rangle.\label{E:secondguy}
        \end{align}
        Now (\ref{E:firstguy}) is absorbed into (\ref{E:firstfriend}) when $\lambda>2$, and (\ref{E:secondguy}) may be absorbed into (\ref{E:secondfriend}) when $\lambda>0$  (in both cases we are left with an order $h$ term).

        As for the second term in (\ref{E:Trouble}$^{\prime}$), it may be absorbed into (\ref{E:commutatorfriend}) as long as $\lambda<3$. All the terms are thus accounted for, completing the proof of (\ref{E:NewerCarlemanEst}) and of the proposition.
    \end{proof}

    \vspace{24pt}

\emph{End of proof of Theorem~\ref{T:uniqDiffIneq}.} We write $U(T,\theta):=u(x_{1},x_{2})$. Since we are only considering $T<T_{0}(\ll 0)$, our hypothesized upper bound (\ref{E:DiffIneq}) gives
\begin{equation*}
    |QU|\leq Ce^{T}(|U|+|\partial_{T}U|+|\partial_{\theta}U|).
\end{equation*}
Now we let $\psi\in C^{\infty}(\Rbb)$ be such that
\begin{equation*}
    \begin{cases}
    \psi=1\quad\text{in }(-\infty,T_{0}-1)\\
    \psi=0\quad\text{in }(T_{0},\infty),
    \end{cases}
\end{equation*}
and we set $$U^{\psi}(T,\theta)=\psi(T)U(T,\theta).$$

The vanishing hypothesis on $u$ says that for every $N$ there exists a constant $C_{N}$ such that $$|u(x)|\leq C_{N}|x|^{N}$$ in a neighborhood of the origin, so that, in the new coordinates, for any $N$ we have $$|U(T,\theta)|\leq C_{N}e^{NT}$$ for $T$ in a neighborhood of $-\infty$. Therefore $$\iint |U^{\psi}|^{2}e^{-NT}d\theta\,dT <\infty$$ for any $N$. The same argument holds for all derivatives of $U^{\psi}$. We then let $\chi\in C^{\infty}(\Rbb)$ be such that
\begin{equation*}
    \begin{cases}
    \chi=0\quad\text{in }(-\infty,-2)\\
    \chi=1\quad\text{in }(-1,\infty),
    \end{cases}
\end{equation*}
and for $R>0$ we let $\chi_{R}(T)=\chi(T/R)$. We may apply Proposition~\ref{P:Carlem} to $\chi_{R}(T)U^{\psi}(T,\theta)$ and take the limit as $R\rightarrow\infty$; by the Dominated Convergence Theorem, Proposition~\ref{P:Carlem} thus holds for $U^{\psi}$.

The righthand side of (\ref{E:CarlemanEst}) is then
\begin{align}
    \iint|h^{2}QU^{\psi}|^{2}e^{-2T/h}d\theta\,dT
    &=h^{4}\iint|\psi QU +\psi^{\prime\prime}U+2\psi^{\prime}\partial_{T}U+c\psi^{\prime}U|^{2}e^{-2T/h}d\theta\,dT\notag\\
    &\leq C h^{4}\iint e^{2T}(|U^{\psi}|^{2}+|\partial_{T}U^{\psi}|^{2}+|\partial_{\theta}U^{\psi}|^{2})e^{-2T/h}d\theta\,dT\label{E:errorend}\\
     &\qquad\qquad\qquad +C h^{4}\iint_{T_{0}-1}^{T_{0}}(|U|^{2}+|\partial_{T}U|^{2})e^{-2T/h}d\theta\,dT.\label{E:decrease}
\end{align}
Since $2T<\epsilon T$, the term (\ref{E:errorend}) is bounded by
\begin{equation*}
    Ch^{2}\iint(|U^{\psi}|^{2}+|h\partial_{T}U^{\psi}|^{2}+|h\partial_{\theta}U^{\psi}|^{2})e^{-2T/h+\epsilon T}d\theta\,dT,
\end{equation*}
and hence can be absorbed into the lefthand side of (\ref{E:CarlemanEst}) when $h>0$ is sufficiently small.

Since $U$ and $\partial_{T}U$ are bounded, the term (\ref{E:decrease}) is bounded by
\begin{equation*}
    Ch^{5}e^{-2(T_{0}-1)/h}.
\end{equation*}
Hence we have
\begin{align*}
    &h\iint\left(|U^{\psi}|^{2}+|h\partial_{T}U^{\psi}|^{2}+|h\partial_{\theta}U^{\psi}|^{2}+|h^{2}\partial_{T}^{2}U^{\psi}|^{2}
            +|h^{2}\partial_{T,\theta}^{2}U^{\psi}|^{2}+|h^{2}\partial_{\theta}^{2}U^{\psi}|^{2}\right)e^{-2T/h+\epsilon T}d\theta\,dT\\
    &\qquad\qquad\qquad\qquad\qquad\qquad \leq Ch^{5}e^{-2(T_{0}-1)/h}.
\end{align*}
Letting $h\rightarrow 0$, we see that $U=0$ when $T<T_{0}-1$, as otherwise the left side grows faster than the right side. Hence the original function $u$ vanishes in a neighborhood of the origin.

We have thus shown that the set of points where $u$ vanishes to infinite order is an open set. The complement is obviously also an open set, so by the connectedness of $X$ we have that $u=0$ in $X$. This concludes the proof of the theorem.
\end{proof}

\end{document}